\documentclass[reqno]{amsart}
\usepackage{subcaption}
\usepackage{graphicx}
\usepackage{array}
\usepackage{hyperref}
\hypersetup{
	pdftitle={Almost every matroid has an M(K4)- or a whirl-minor},
	pdfauthor={Jorn van der Pol}
}

\newcommand{\whirl}{\ensuremath{\mathcal{W}^3}}
\newcommand{\PG}[2]{\mathrm{PG}(#1,#2)}

\newcommand{\mX}{m_{\mathcal{X}}}
\newcommand{\sX}{s_{\mathcal{X}}}
\newcommand{\pX}{p_{\mathcal{X}}}
\newcommand{\rs}[1]{\mathrm{rs}(#1)}

\newcommand{\exlin}[2]{\mathrm{ex}_{\mathrm{lin}}(#1,#2)}
\newcommand{\exlinind}[2]{\mathrm{ex}_{\mathrm{lin}}^{\mathrm{ind}}(#1,#2)}

\newtheorem{theorem}{Theorem}
\newtheorem{conjecture}[theorem]{Conjecture}
\newtheorem{lemma}[theorem]{Lemma}
\newtheorem*{question}{Question}

\author{Jorn van der Pol}
\address{Department of Combinatorics and Optimization, University of Waterloo, Waterloo, Ontario, Canada}
\title{Almost every matroid has an $M(K_4)$- or a \whirl-minor}

\begin{document}
\maketitle

\begin{abstract}
	We show that almost every matroid contains the rank-3 whirl \whirl\ or the complete-graphic matroid $M(K_4)$ as a minor.
\end{abstract}

\section{Introduction}

A matroid is called sparse paving if and only if its nonspanning circuits are hyperplanes as well. Alternatively, a matroid is sparse paving if and only if no nonbasis (a nonspanning set whose cardinality is equal to the rank of the matroid) can be transformed into another by exchanging a single element for another.

Sparse paving matroids, relatively benign objects compared to general matroids, play an important role in matroid enumeration problems, as it is widely believed that almost every matroid is sparse paving (cf.\ \cite{MayhewNewmanWelshWhittle2011} and references therein).
Mayhew, Newman, Welsh and Whittle~\cite{MayhewNewmanWelshWhittle2011} conjectured that sparse paving matroids are ubiquitous in another sense as well, namely that any fixed sparse paving matroid is contained as a minor in almost every matroid.
\begin{conjecture}[Mayhew--Newman--Welsh--Whittle, \cite{MayhewNewmanWelshWhittle2011}]\label{conj:N-minor}
	Let $N$ be a sparse paving matroid. Almost every matroid has an $N$-minor; i.e.\ the fraction of matroids on ground set~$E$ that do not have an $N$-minor tends to~0 as $|E|\to\infty$.
\end{conjecture}
Various special cases of Conjecture~\ref{conj:N-minor} have been verified; including the case where~$N$ is a uniform matroid~\cite{PendavinghVanderpol2018} or one of the six-element matroids $P_6$, $Q_6$ and $R_6$~\cite{PendavinghVanderpol2015}.\protect\footnote{Throughout this paper, matroid terminology and names for named matroids follow~\cite{Oxley2011}.} In particular, Conjecture~\ref{conj:N-minor} has been verified for each sparse paving matroid of rank~3 on 6~elements, except the whirl \whirl\ and the complete graphic matroid $M(K_4)$ (see Figure~\ref{fig:whirl-k4}). In this paper, we prove that almost every matroid contains at least of these as a minor.
\begin{theorem}\label{thm:main-matroid}
	Almost every matroid has a \whirl- or $M(K_4)$-minor.
\end{theorem}
\begin{figure}[h]\centering
	\subcaptionbox{The whirl $\whirl$.}[.4\linewidth]{\includegraphics[width=3cm]{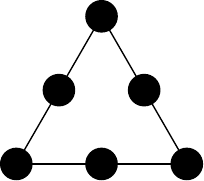}}
	\subcaptionbox{The graphic matroid $M(K_4)$.}[.4\linewidth]{\includegraphics[width=3cm]{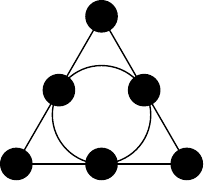}}
	\caption{\label{fig:whirl-k4}The two matroids that appear in Theorem~\ref{thm:main-matroid}.}
\end{figure}

At the heart of the proof of Theorem~\ref{thm:main-matroid} lies an analysis of the number of rank-3 matroids that do not have \whirl\ or $M(K_4)$ as a restriction.
\begin{theorem}\label{thm:main-rank3}
	The number of rank-3 $\{\whirl, M(K_4)\}$-free matroids on ground set~$[n]$ is $2^{o(n^2)}$.\protect\footnote{The set $[n]$ denotes $\{1,2,\ldots,n\}$. Asymptotic notation will always refer to the regime $n\to\infty$. Logarithms are taken with respect to base 2.}
\end{theorem}
A version of Theorems~\ref{thm:main-matroid} and~\ref{thm:main-rank3} for sparse paving matroids was proved earlier on the Matroid Union blog~\cite{Vanderpol2020}.

The remainder of this paper is structured as follows. In Section~\ref{sec:enum}, we review known results about matroid enumeration and prove our main technical tool, which allows us to bound the size of a class of matroids in terms of its low-rank members. In Section~\ref{sec:s}, we prove the sparse paving versions of Theorems~\ref{thm:main-matroid} and~\ref{thm:main-rank3}, after which both theorems are proved in Section~\ref{sec:proofs}. Finally, in Section~\ref{sec:fano}, we comment on replacing $M(K_4)$ by $F_7$ in the two main theorems of this paper.

\section{\label{sec:enum}Matroid enumeration}

In this section, we review known results on matroid enumeration and prove a lemma which allows us to conclude that a contraction-closed class~$\mathcal{M}$ of matroids is ``small'' based on the number of rank-$s$ matroids that it contains.

\subsection{Notation}

Throughout this paper, $m(n)$ and $s(n)$ denote the number of matroids and sparse paving matroids on ground set~$[n]$, respectively. Similarly, $m(n,r)$ and $s(n,r)$ are the number of matroids (sparse paving matroids) on ground set~$[n]$ of rank~$r$.

When talking about classes of matroids, we will use subscripts for the corresponding numbers in this class; for example, $s_{\mathcal{M}}(n,r)$ denotes the number of rank-$r$ sparse paving matroids on ground set~$[n]$ in the class~$\mathcal{M}$.

We say that \emph{almost every matroid} satisfies property~$P$ if the class~$\mathcal{M}$ of matroids that do not satisfy~$P$ is small in the sense that $m_{\mathcal{M}}(n) = o(m(n))$; in practice, this often means showing that $m_{\mathcal{M}}(n) = o(s(n))$.

In this paper, the main class of interest is $\mathcal{M} = \mathcal{X}$, where $\mathcal{X} = \text{Ex}(\whirl, M(K_4))$ is the class of $\{\whirl, M(K_4)\}$-free matroids; thus, Theorem~\ref{thm:main-matroid} is equivalent to the statement that $\mX(n) = o(m(n))$, while Theorem~\ref{thm:main-rank3} states that $\mX(n,3) = 2^{o(n^2)}$.

\subsection{Lower bound}

The following lower bound on the number of (sparse paving) matroids on ground set~$[n]$ is due to a construction by Graham and Sloane~\cite{GrahamSloane1980}.
\begin{lemma}\label{lemma:s-lowerbound}
	$\log m(n) \ge \log s(n) \ge \frac{1}{n}\binom{n}{n/2}$.
\end{lemma}

\subsection{Bounding a class in terms of its rank-$s$ members}

The following lemma is a straightforward extension of the main technical result in~\cite{BansalPendavinghVanderpol2014}, and first appeared on the Matroid Union blog~\cite{PendavinghVanderpol2016}.
\begin{lemma}\label{lemma:entropy-blowup1}
	Let $\mathcal{M}$ be a class of matroids that is closed under contraction. For all $t \le r \le n$,
	\begin{equation*}
		\frac{\log(1 + m_\mathcal{M}(n,r))}{\binom{n}{r}} \le \frac{\log(1+m_\mathcal{M}(n-t,r-t))}{\binom{n-t}{r-t}}.
	\end{equation*}
\end{lemma}

Let $s \ge 3$ be a fixed integer. Applying Lemma~\ref{lemma:entropy-blowup1} with $t = r-s$ provides an upper bound on $m_\mathcal{M}(n,r)$ in terms of $m_\mathcal{M}(n-r+s,s)$. Note that the upper bound provided by the next lemma is in terms of~$s(n)$; this facilitates use of the lemma to prove statements of the form ``almost every sparse paving matroid satisfies~$P$'' as well as those of the form ``almost every matroid satisfies~$P$''.
\begin{lemma}\label{lemma:entropy-blowup2}
	Let $\mathcal{M}$ be a class of matroids that is closed under contraction, and let $s \ge 3$ be an integer. If there exist $c < 1/2$ and $n_0$ such that $\log m_{\mathcal{M}}(n,s) \le \frac{c}{n}\binom{n}{s}$ for all $n \ge n_0$, then $m_{\mathcal{M}}(n) = o(s(n))$; in particular, then almost every matroid is not in $\mathcal{M}$.
\end{lemma}
\begin{proof}
	Let $c'$ and $c''$ be such that $c < c' < c'' < 1/2$. For positive integers $n$, define $R_n = \left[\frac{n}{2}-\sqrt{n}, \frac{n}{2}+\sqrt{n}\right] \cap \mathbb{Z}$ and $R_n^c = \{0, 1, \ldots, n\}\setminus R_n$. Let $n \ge n_0$ be so large that $\frac{n}{2}-\sqrt{n} > s$, $c'/(1-2/\sqrt{n}) < c''$, and
	\begin{equation*}
		1 + 2^{\frac{c}{n-r+1}\binom{n-r+s}{s}}
			\le 2^{\frac{c'}{n-r+s}\binom{n-r+s}{s}}
			\qquad\text{for all $r \in R_n$}.
	\end{equation*}
	
	For $r \in R_n$, an application of Lemma~\ref{lemma:entropy-blowup1} gives
	\begin{equation}\label{eq:entropy-blowup2-1}
		\log m_\mathcal{M}(n,r)
			\le \frac{c'}{n-r+s}\binom{n}{r}
			\le \frac{c'}{\frac{n}{2} - \sqrt{n}}\binom{n}{n/2}
			\le \frac{2c''}{n}\binom{n}{n/2}.
	\end{equation}
	
	By~\cite[Theorem~16]{PendavinghVanderpol2015b} (and the observation that $m_{\mathcal{M}}(n,r) \le m(n,r)$),
	\begin{equation}\label{eq:entropy-blowup2-2}
		\sum_{r \in R_n^c} m_{\mathcal{M}}(n,r) \le \sum_{r \in R_n^c} m(n,r) = o(s(n)).\protect\footnote{Although~\cite[Theorem~16]{PendavinghVanderpol2015b} states only that $\sum_{r \in R_n^c} m(n,r) = o(m(n))$, its proof implies the statement that is used here.}
	\end{equation}
	
	Combining~\eqref{eq:entropy-blowup2-1} and~\eqref{eq:entropy-blowup2-2},
	\begin{equation*}
		\begin{array}{*7{@{\,}>{\displaystyle}c@{\,}}}
			m_{\mathcal{M}}(n)
				&=& \sum_{r \in R_n} m_{\mathcal{M}}(n,r)
					&+& \sum_{r \in R_n^c} m_{\mathcal{M}}(n,r) \\
				&\le& (2\sqrt{n}+1) 2^{\frac{1-\delta}{n}\binom{n}{n/2}} &+& o(s(n)) &=& o(s(n)),
		\end{array}
	\end{equation*}
	where the final step follows from the lower bound on $s(n)$ from Lemma~\ref{lemma:s-lowerbound}.
\end{proof}

\section{\label{sec:s}Sparse paving matroids without $M(K_4)$- or \whirl-minor}

In this section, we prove the sparse paving versions of Theorems~\ref{thm:main-matroid} and~\ref{thm:main-rank3}. The contents of this section appeared before on the Matroid Union blog~\cite{Vanderpol2020}.
\begin{theorem}\label{thm:sX}
	Almost every sparse paving matroid contains \whirl\ or $M(K_4)$ as a minor, i.e.\ $\sX(n) = o(s(n))$.
\end{theorem}

Our starting point will be sparse paving matroids of rank~3 on ground set~$[n]$, where we assume that $n \ge 4$. Pairs of circuit-hyperplanes of such matroids intersect in at most one point. Conversely, any collection of 3-element subsets of~$[n]$ that pairwise intersect in at most one point forms the collection of circuit-hyperplanes of a rank-3 sparse paving matroid. Thus, sparse paving matroids of rank-3 can be thought of as linear 3-uniform hypergraphs.

A rank-3 sparse paving matroid is $\{\whirl, M(K_4)\}$-free if and only if its corresponding hypergraph has the property that the subgraph induced by any six vertices spans at most two edges (equivalently, such a hypergraph does not contain a linear 3-cycle, i.e.\ the hypergraph corresponding to \whirl, as a subgraph). We call such a hypergraph a \emph{Ruzsa--Szemer\'edi hypergraph}, and write $\rs{n}$ for the maximum number of edges in a Rusza--Szemer\'edi hypergraph on $n$ vertices. The following result is known as the $(6,3)$-theorem.
\begin{theorem}[Ruzsa--Szemer{\'e}di, \cite{RuzsaSzemeredi1978}]\label{thm:RS}
	$\rs{n} = o(n^2)$.
\end{theorem}

The following result, a counting version of the $(6,3)$-theorem, is by no means new; in the context Ruzsa--Szemer\'edi hypergraphs, it was mentioned by Balogh and Li~\cite{BaloghLi2020} as an extension of a result of Erd\H{o}s, Frankl and R\"{o}dl~\cite{ErdosFranklRodl1986}. We include a proof for the sake of completeness.
\begin{lemma}\label{lemma:sX-3}
	$\sX(n,3) = 2^{o(n^2)}$.
\end{lemma}
\begin{proof}
	Let~$H$ be a Ruzsa--Szemer\'edi hypergraph on~$n$ vertices, and let~$\partial H$ be its 2-shadow, i.e.\ the (ordinary) graph on the same vertex set in which two vertices are adjacent if and only if they appear in the same edge of~$H$.
	
	Since~$H$ does not contain a linear 3-cycle, $\partial H$ has the property that each edge appears in a unique triangle. This implies that $H \mapsto \partial H$ is injective.
	
	As $H$ has at most $\rs{n}$ edges, $\partial H$ has at most $3\rs{n}$ edges. It follows that the number of possible $\partial H$, and hence the number of Ruzsa--Szemer\'edi hypergraphs, on~$n$ vertices is at most
	\begin{equation*}
		\sum_{i \le 3\rs{n}} \binom{\binom{n}{2}}{i} \le 2^{H\left(3\rs{n}/\binom{n}{2}\right)\binom{n}{2}}.
	\end{equation*}
	Here, $H$ denotes the binary entropy function, which has the property that $H(\varepsilon) \downarrow 0$ as $\varepsilon \downarrow 0$. The lemma now follows as $\rs{n}/\binom{n}{2} = o(1)$ by Theorem~\ref{thm:RS}.
\end{proof}

\begin{proof}[Proof of Theorem~\ref{thm:sX}]
	Let $\mathcal{M}$ be the class of sparse paving matroids in $\mathcal{X}$, so we can write $\sX(n,3) = m_\mathcal{M}(n,3)$. By Lemma~\ref{lemma:sX-3}, $\log m_\mathcal{M}(n,3) \le \frac{0.49}{n}\binom{n}{3}$ for sufficiently large~$n$. It follows from Lemma~\ref{lemma:entropy-blowup2} that $\sX(n) = m_\mathcal{M}(n) = o(s(n))$.
\end{proof}

Due to a construction of Behrend's~\cite{Behrend1946}, the exponent~2 in the statement of Theorem~\ref{thm:RS} can not be replaced by $2-\varepsilon$ for any $\varepsilon > 0$ (see~\cite{RuzsaSzemeredi1978}). Since the property of being a Ruzsa--Szemer\'edi hypergraph can not be destroyed by removing edges, the same construction implies that $\sX(n,3) = 2^{\Omega(n^{2-\varepsilon})}$ for all $\varepsilon > 0$.

\section{\label{sec:proofs}From sparse paving to general matroids: Proof of Theorem~\ref{thm:main-matroid}}

We now turn to proving the main result, Theorem~\ref{thm:main-matroid}.

\subsection{From sparse paving to paving matroids}

A paving matroid is one in which the only interesting flats are the hyperplanes: all smaller flats are independent. Equivalently, a matroid~$M$ is paving if and only if each of its circuits has cardinality~$r(M)$ or $r(M)+1$. Every sparse paving matroid is paving, but there are paving matroids that are not sparse paving.

The goal of this section is to extend the bound on~$\sX(n,3)$ in Lemma~\ref{lemma:sX-3} to paving matroids. We use a technique by Pendavingh and Van der Pol~\cite{PendavinghVanderpol2017} to encode each rank-3 paving matroid as a pair of rank-3 sparse paving matroids (the technique works in greater generality, but here it suffices to specialise to the rank-3 case). If the paving matroid is $\{\whirl, M(K_4)\}$-free, then so are the two sparse paving matroids, which will show that $\pX(n,3)$ is at most $(\sX(n,3))^2$.

Let~$M$ be a rank-3 paving matroid on a ground set~$E$ that is linearly ordered. The matroid~$M$ can be reconstructed from the collection
\begin{equation*}
	\mathcal{V}(M) = \negthickspace\bigcup_{H \in \mathcal{H}(M)}\negthickspace \mathcal{V}(H),
\end{equation*}
where $\mathcal{H}(M)$ is the set of hyperplanes of~$M$, and for each hyperplane~$H$, the elements of~$\mathcal{V}(H)$ are the consecutive triples in~$H$:
\begin{equation*}
	\mathcal{V}(H) = \left\{V\!\in\!\binom{H}{3} : \text{there are no $v, v' \in V$ and $h \in H\setminus V$ such that $v < h < v'$}\right\}.
\end{equation*}
The linear order on~$E$ induces a linear order on~$\mathcal{V}(H)$, so that we can write $\mathcal{V}(H) = \{V_H^0, V_H^1, \ldots, V_H^{|H|-2}\}$ such that $V_H^i < V_H^j$ for all $i < j$ and $|V_H^i \cap V_H^j| = 2$ if and only if $|i-j|=1$.

Define
\begin{equation*}
	\mathcal{V}^0(H) = \{V_H^i : \text{$i$ is even}\}, \qquad \mathcal{V}^1(H) = \{V_H^i : \text{$i$ is odd}\},
\end{equation*}
and
\begin{equation*}
	\mathcal{V}^0(M) = \negthickspace\bigcup_{H \in \mathcal{H}(M)}\negthickspace \mathcal{V}^0(H), \qquad \mathcal{V}^1(M) = \negthickspace\bigcup_{H \in \mathcal{H}(M)}\negthickspace \mathcal{V}^1(H).
\end{equation*}
\begin{theorem}[Pendavingh--Van der Pol, \cite{PendavinghVanderpol2017}]\label{thm:paving-encoding}
	For $k \in \{0,1\}$, the set $\mathcal{V}^k(M)$ is the set of circuit-hyperplanes of a rank-3 sparse paving matroid on $E$. Moreover, the map $M \mapsto (\mathcal{V}^0(M), \mathcal{V}^1(M))$ is injective.
\end{theorem}
We will refer to the two sparse paving matroids corresponding to~$M$ as~$M^0$ and~$M^1$. The original matroid $M$ is a weak-map image of both~$M^0$ and~$M^1$ (i.e.\ if $D$ is dependent in $M^0$ or in $M^1$, then it is dependent in $M$); as $M$ is paving, this implies the following lemma.
\begin{lemma}\label{lemma:paving-whirl-k4-free}
	If~$M$ is $\{\whirl, M(K_4)\}$-free, then so are $M^0$ and $M^1$.
\end{lemma}

The following lemma is the paving version of Theorem~\ref{thm:main-rank3}.
\begin{lemma}\label{lemma:pX-3}
	$\pX(n,3) = 2^{o(n^2)}$.
\end{lemma}
\begin{proof}
	By Theorem~\ref{thm:paving-encoding} and Lemma~\ref{lemma:paving-whirl-k4-free}, every rank-3 $\{\whirl, M(K_4)\}$-free paving matroid on ground set~$[n]$ can be described by a pair of rank-3 $\{\whirl, M(K_4)\}$-free sparse paving matroids on~$[n]$. This immediately implies $\pX(n,3) \le \left(\sX(n,3)\right)^2$. The conclusion now follows from Lemma~\ref{lemma:sX-3}.
\end{proof}

\subsection{From paving to general matroids}

We are now ready to prove Theorem~\ref{thm:main-rank3}, the upper bound on~$\mX(n,3)$.

Any rank-3 matroid~$M$ on ground set~$[n]$ with $k$ rank-1 flats can be described by a rank-3 paving matroid on ground set~$[k]$ (that is isomorphic to the simplification of~$M$), together with an ordered set partition $(X_0, X_1, \ldots, X_k)$ of~$[n]$, such that $X_0$ gives the set of loops of~$M$, and $X_1, X_2, \ldots, X_k$ describe the parallel classes of~$M$. If $M$ is $\{\whirl, M(K_4)\}$-free, then so is its associated paving matroid.
\begin{proof}[Proof of Theorem~\ref{thm:main-rank3}]
	The number of ordered partitions of~$[n]$ into $k+1$ parts is at most $n^n$. Summing over~$k$, we obtain
	\begin{equation*}
		\mX(n,3) \le \sum_{k=3}^n n^n \pX(k,3) \le n^{n+1} \pX(n,3).
	\end{equation*}
	As $n^{n+1} = 2^{o(n^2)}$, the result follows from the bound on $\pX(n,3)$ in Lemma~\ref{lemma:pX-3}.
\end{proof}

Finally, we prove Theorem~\ref{thm:main-matroid}, which is readily implied by the following result.
\begin{theorem}
	$\mX(n) = o(s(n))$.
\end{theorem}
\begin{proof}
	By Theorem~\ref{thm:main-rank3}, $\log \mX(n,3) \le \frac{0.49}{n}\binom{n}{3}$ for sufficiently large $n$. The result now follows from an application of Lemma~\ref{lemma:entropy-blowup2}.
\end{proof}

\section{\label{sec:fano}Does almost every matroid contain \whirl\ or $F_7$ as a minor?}

It is natural to ask if the following strengthening of Theorem~\ref{thm:main-matroid} holds: Does almost every matroid have \whirl\ or $F_7$ as a minor? In this section, we show that the corresponding version of the Ruzsa--Szemer\'{e}di $(6,3)$-theorem fails.

As in Section~\ref{sec:s}, it will be useful to think of rank-3 sparse paving matroids as linear 3-uniform hypergraphs.

Let $\mathcal{G}$ be a family of 3-uniform linear hypergraphs. In this section, we write $\exlin{n}{\mathcal{G}}$ for the \emph{linear Tur\'{a}n number}, i.e.\ the maximum number of edges in an $n$-vertex linear 3-uniform hypergraph that does not contain a copy of any member of $\mathcal{G}$ as a subgraph. Similarly, we write $\exlinind{n}{\mathcal{G}}$ for the \emph{linear induced Tur\'{a}n number}: the maximum number of edges in an $n$-vertex 3-uniform linear hypergraph that does not contain a copy of any member of $\mathcal{G}$ as an \emph{induced} subgraph. When $\mathcal{G} = \{G\}$ has but a single member, we shall write $\exlin{n}{G} = \exlin{n}{\mathcal{G}}$. In these terms, the Ruzsa--Szemer\'{e}di $(6,3)$-theorem states that
\begin{equation*}
	\exlinind{n}{\{\whirl, M(K_4)\}} = \exlin{n}{\whirl} = o(n^2).
\end{equation*}

Let $F$ be the \emph{3-fan} or \emph{sail}; i.e.\ the hypergraph on vertices $\{1,2,3,4,5,6,7\}$ with edges $\{1,2,3\}$, $\{1,4,5\}$, $\{1,6,7\}$ and $\{3,5,7\}$ (see Figure~\ref{fig:fan}).

\begin{figure}[t]
	\subcaptionbox{\label{fig:fan}The 3-fan or sail $F$.}[.4\linewidth]{\includegraphics[width=3cm]{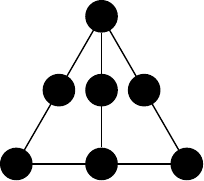}}
	\subcaptionbox{\label{fig:fano}The Fano plane $F_7$.}[.4\linewidth]{\includegraphics[width=3cm]{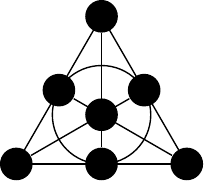}}
	\caption{Two hypergraphs that appear in Section~\ref{sec:fano}}
\end{figure}

\begin{theorem}[F\"{u}redi--Gy\'{a}rf\'{a}s, \cite{FurediGyarfas2017}]\label{thm:fan}
	$\exlin{n}{F} \le \frac{n^2}{9}$.
\end{theorem}

For $r \ge 2$, let $B_r$ be the linear 3-uniform hypergraph on vertices $\{x \in \mathbb{F}_2^r : (x_1,x_2) \neq (0,0)\}$, in which three distinct vertices $x$, $y$ and $z$ form a hyperedge if and only if $x + y + z = 0$ in $\mathbb{F}_r^2$. Note that the vertices and hyperedges of this hypergraph are formed by the points and lines of a rank-$r$ binary Bose--Burton geometry, which is obtained from $\PG{r-1}{2}$ by removing the points of a full subgeometry of rank $r-2$. Alternatively, the edges of $B_r$ form the blocks of a transversal design with three groups, where two vertices of $B_r$ are in the same group when they coincide on their first two coordinates.

The hypergraph $B_r$ has $n = 3\cdot 2^{r-2}$ vertices and $2^{2(r-2)} = \frac{n^2}{9}$ edges and was shown by F\"{u}redi and Gy\'{a}rf\'{a}s~\cite{FurediGyarfas2017} to be maximal $F$-free. Using the geometric construction of~$B_r$, it can be shown that these hypergraphs are not only $F$-free, but induced-$\{\whirl, F_7\}$-free as well; the following result shows that they are in fact maximal induced-$\{\whirl, F_7\}$-free.
\begin{theorem}
	$\exlinind{n}{\{\whirl, F_7\}} \le \frac{n^2}{9}$. Equality holds infinitely often.
\end{theorem}
\begin{proof}
	Let $H$ be a linear 3-uniform hypergraph that contains neither \whirl\ nor $F_7$ as an induced subgraph. We claim that $H$ does not contain the 3-fan $F$ as a subgraph. Suppose, for the sake of contradiction, that $H$ contains edges $\{v_1, v_2, v_3\}$, $\{v_1, v_4, v_5\}$, $\{v_1, v_6, v_7\}$ and $\{v_3, v_5, v_7\}$, where the seven vertices are distinct. Let $V = \{v_1, v_2, \ldots, v_7\}$ and let $S = \{2,4,6\}$. For each $s \in S$, the subgraph $H[V\setminus\{v_s\}]$ contains a copy of \whirl; as such a subgraph cannot be induced, $V$ spans each of the edges $\{v_2, v_4, v_6\} \triangle \{v_s, v_{s+1}\}$, $s \in S$. But this means that $H$ contains a copy of $F_7$: a contradiction, so $H$ does not contain a copy of $F$. It follows that
	\begin{equation*}
		\exlinind{n}{\{\whirl, F_7\}} \le \exlin{n}{F} \le \frac{n^2}{9},
	\end{equation*}
	where the final inequality follows from Theorem~\ref{thm:fan}.
	
	For each $r \ge 2$, the hypergraph $B_r$ provides an example for which the upper bound is attained.
\end{proof}

Let $f(n) = s_{\text{Ex}(\whirl, F_7)}(n,3)$ be the number of linear 3-uniform hypergraphs without induced \whirl\ or $F_7$. A trivial upper bound on $f(n)$ is
\begin{equation*}
	f(n) \le \sum_{i=0}^{\frac{n^2}{9}} \binom{\binom{n}{3}}{i} = 2^{O(n^2\log n)}.
\end{equation*}
We end this section with a question.
\begin{question}
	Is $f(n) = 2^{\Theta(n^2)}$?
\end{question}
A sufficiently strong upper bound on $f(n)$ may be a first step toward proving that almost every matroid has a \whirl- or $F_7$-minor. However, even in that case, additional ideas are required, as the $\{\whirl, F_7\}$-version of Lemma~\ref{lemma:paving-whirl-k4-free} fails.

\bibliographystyle{alpha}
\bibliography{whirlandk4}
\end{document}